\documentclass[11pt,hyp,]{nyjm}
\usepackage{hyperref}
\hypersetup{nesting=true,debug=true,naturalnames=true}
\usepackage{graphicx,amssymb,upref}

\let\<\langle
\let\>\rangle

\let\uml\"

\usepackage{comment}
\usepackage[justification=centering]{caption}
\usepackage{longtable, makecell, multirow, tabularx}
\usepackage[arc,all]{xy}

\theoremstyle{plain}

\newtheorem{lem}[subsection]{Lemma}
\newtheorem{thm}[subsection]{Theorem}
\newtheorem{prop}[subsection]{Proposition}

\theoremstyle{definition}
\newtheorem{defn}[subsection]{Definition}
\newtheorem{rem}[subsection]{Remark}
\newtheorem{notation}[subsection]{Notations}

\theoremstyle{plain} \newtheorem*{thm*}{Theorem}
\newtheoremstyle{TheoremNum}
{8pt}{8pt}              
{\itshape}                      
{}                              
{\bfseries}                     
{.}                             
{ 5pt plus 1pt minus 1pt }                             
{\thmname{#1}\thmnote{ \bfseries #3}}
\theoremstyle{TheoremNum}
\newtheorem{thmrep}{Theorem}

\title[Some nontrivial secondary Adams differentials]{Some Nontrivial Secondary Adams Differentials On the Fourth Line}

\author[X. Wang]{Xiangjun Wang}
\address{Department of Mathematics, Nankai University, No.94 Weijin Road, Tianjin 300071, P. R. China}
\email{xjwang@nankai.edu.cn}

\author[Y. Wang]{Yaxing Wang}
\address{Department of Mathematics, Nankai University, No.94 Weijin Road, Tianjin 300071, P. R. China}
\email{yxwangmath@163.com}

\author[Y. Zhang]{Yu Zhang}
\address{Department of Mathematics, Nankai University, No.94 Weijin Road, Tianjin 300071, P. R. China}
\email{zhang.4841@buckeyemail.osu.edu}



\thanks{The authors are supported by the National Natural Science Foundation of China (No. 12271183).  The third named author is also supported by the National Natural Science Foundation of China (No. 12001474; 12261091).}

\begin{document} 
 
\begin{abstract}  
Let $p \geq 5$ be an odd prime. Using the correspondence between secondary Adams differentials and secondary algebraic Novikov differentials, we compute four families of nontrivial secondary differentials on the fourth line of the Adams spectral sequence. We also recover all secondary differentials on the first three lines of the Adams spectral sequence.
\end{abstract} 
\maketitle
\tableofcontents


\section{Introduction}
\label{sec:intro}

The Adams spectral sequence (ASS) is one of the most useful tools to compute the stable homotopy groups of the sphere $\pi_*(S)$. The ASS has $E_2$-page $Ext_{\mathcal{A}_{*}}^{*,*} (\mathbb{F}_p, \mathbb{F}_p)$, where $\mathcal{A}_{*}$ is the dual mod $p$ Steenrod algebra. 

In this paper, we always assume $p$ is an odd prime.  Then we have 
$$\mathcal{A}_* = P [\xi_1, \xi_2, \cdots] \otimes E[\tau_0, \tau_1, \tau_2, \cdots]$$
where $P [\xi_1, \xi_2, \cdots]$ is a polynomial algebra with coefficients in $\mathbb{F}_p$, and $E [\tau_0, \tau_1, \tau_2, \cdots]$ is an exterior algebra with coefficients in $\mathbb{F}_p$.

The Adams-Novikov spectral sequence (ANSS) is another useful tool for computing  $\pi_*(S)$. The ANSS has $E_2$-page
$Ext_{BP_*BP}^{*,*}(BP_*, BP_*)$, where $BP$ denotes the Brown-Peterson spectrum. We have 
\begin{equation*}
   BP_*:= \pi_*(BP) = \mathbb{Z}_{(p)}[v_1, v_2, \cdots], \quad  BP_*BP = BP_* [t_1,t_2,\cdots]
\end{equation*}
where $\mathbb{Z}_{(p)}$ denotes the integers localized at $p$.

The Adams-Novikov $E_2$-page can be computed via the
algebraic Novikov spectral sequence (algNSS) \cite{Miller_algANSS,Novikov_methods}.  The $E_2$-page of the algNSS has the form $Ext^{s,t}_{P_*} (\mathbb{F}_p, I^k / I^{k+1})$, where $I$ denotes the ideal $(p, v_1, v_2, \cdots) \subset BP_*$, and $P_* = BP_* BP/I = P[t_1, t_2, \cdots]$ is the $\mathbb{F}_{p}$-coefficient polynomial algebra.  Here, we have re-indexed the pages to align with the notations in Gheorghe-Wang-Xu \cite{Gheorghe_Wang_Xu} and Isaksen-Wang-Xu \cite{Isaksen_Wang_Xu_more_stable_stems}.

The $E_2$-page of the Adams spectral sequence can also be computed via another spectral sequence, called the Cartan-Eilenberg spectral sequence (CESS) \cite{Cartan_Eilenberg, Ravenel_stable_homotopy_groups}.  For odd prime $p$, the $E_2$-page of the CESS coincides with the $E_2$-page of the algNSS.
Then, we have the following diagram of spectral sequences.
\begin{align*}
\xymatrix@C=5em{
Ext_{P_*}^{s,t} (\mathbb{F}_{p}, I^k / I^{k+1}) \ar@{=>}[r]^{CESS} \ar@{=>}[d]_{algNSS} & Ext_{\mathcal{A}_*}^{s+k,t+k} (\mathbb{F}_p , \mathbb{F}_p) \ar@{=>}[d]^{ASS}  \\ 
Ext^{s,t}_{BP_*BP}(BP_*, BP_*) \ar@{=>}[r]^{\quad \quad \quad   ANSS}  & \pi_{t-s} (S) 
}
\end{align*}

In practice, the main difficulty of computing with the ASS is that the Adams differentials $d^{Adams}_r$'s are difficult to be determined in general.  On the other hand, the algebraic Novikov differentials $d^{alg}_r$'s are much easier to be computed.  This is because the entire construction of the algNSS is purely algebraic.  Computing $d^{alg}_r$'s does not require any topological background knowledge.  It turns out that when $r = 2$, there is a direct correspondence between $d^{Adams}_2$'s and $d^{alg}_2$'s.

\begin{thm}[Novikov \cite{Novikov_methods}, Andrews-Miller \cite{Andrews_Miller,Miller_relations}]
\label{thm: miller transition}
Let $z \in Ext_{\mathcal{A}_*}^{s+k,t+k} (\mathbb{F}_p , \mathbb{F}_p)$ be a nontrivial element detected in the CESS by $x \in Ext^{s,t}_{P_*}(\mathbb{F}_{p}, I^k / I^{k+1})$.  Regard $x$ as an element in the algNSS, then the secondary algebraic Novikov differential $d^{alg}_2 (x)$ represents the secondary Adams differential $d^{Adams}_2 (z)$. 
\end{thm}

Let $p \geq 5$.  A complete list of generators together with their $d^{Adams}_2 (z)$ has been determined for the first three lines of the Adams $E_2$-page, i.e.  $Ext_{\mathcal{A}_{*}}^{s,t} (\mathbb{F}_p, \mathbb{F}_p)$ with $s = 1, 2, 3$ (see \cite{Aikawa_3_dimensional, Liulevicius_factorization, Miller_Ravenel_Wilson, Shimada_Yamanoshita, Wang_third_line_adams,Zhao_Wang_adams}).  Meanwhile, only partial results are known for the fourth line $Ext_{\mathcal{A}_{*}}^{4,*} (\mathbb{F}_p, \mathbb{F}_p)$ (see, for example, \cite{Zhong_Hong_Zhao}). 

In this paper, we demonstrate a practical computing strategy to determine $d^{Adams}_2$'s by computing their corresponding $d^{alg}_2$'s. We will work on several explicit examples and provide detailed proof.  Our main result is the following.

\begin{thmrep}[\ref{thm: new adams d2 fourth line}]
There are nontrivial secondary Adams differentials given as follows:
\begin{enumerate}
    \item $d_2^{Adams} (h_{4, i} h_{3, i} g_i) = a_0 b_{4,i-1} h_{3, i} g_i$, for $i \geq 1$.
    \item $d_2^{Adams} (h_{4, i} h_{3, i+1} k_{i+2}) = a_0 b_{4,i-1} h_{3, i+1} k_{i+2}$, for $i \geq 1$.
    \item $d_2^{Adams} (h_{4, i} g_i h_{i+3}) = a_0 b_{4,i-1} g_i h_{i+3}$, for $i \geq 1$.
    \item $d_2^{Adams} (h_{3, i} h_{2, i+1} k_{i}) = a_0 b_{3,i-1} h_{2, i+1} k_{i}$, for $i \geq 1$.
\end{enumerate}
\end{thmrep}

\begin{rem}
    Here, we follow the conventions of \cite{Wang_third_line_adams, Zhong_Hong_Zhao} to name Adams $E_2$-page elements by their May spectral sequence (MSS) representatives, compare with Table \ref{table: first and second line basis} and Table \ref{table: third line basis}.  We would like to comment more explicitly on the indeterminacy of these classes. For example, the result of (1) should be interpreted as follows.  If an element $x \in Ext_{\mathcal{A}_{*}}^{4,*} (\mathbb{F}_p, \mathbb{F}_p)$  has MSS representative $h_{4, i} h_{3, i} g_i:= h_{4, i} h_{3, i} h_{2, i} h_{1, i}$, then its secondary Adams differential $d_2^{Adams} (x)$ has MSS representative $a_0 b_{4,i-1} h_{3, i} g_i:= a_0 b_{4,i-1} h_{3, i} h_{2, i} h_{1, i}$. More details of the MSS are reviewed in Section \ref{sec: spectral sequence}.
\end{rem}

It is straightforward to verify that these four families of elements are indecomposable, i.e., they can not be written as products of elements from the first three lines. Consequently, one can not deduce the differentials simply via Leibniz rule.

From our point of view, the practical computational strategy here is possibly more interesting than the result itself.  To further demonstrate this, in Section \ref{sec: first three lines}, we use the same strategy to recover all secondary Adams differentials on the first three lines.

Previously, the nontrivial Adams differentials on the third line were computed in \cite{Wang_third_line_adams} using matrix Massey products \cite{May_massey_product}.  Comparatively, our computation has the following advantages: (i) Our computations can be easily adapted to analyze other $d^{Adams}_2$'s of interest.   On the contrary, the matrix Massey product method could fail when the relevant indeterminacy is nontrivial; (ii) Our computations of the algebraic Novikov differentials are routine and purely algebraic.  Such computations are comparatively more straightforward than the previous ones using matrix Massey products.

\subsection*{Organization of the paper} 
In Section \ref{sec: hopf algebroid}, we review the algebraic structures and constructions related to Hopf algebroids. These structures are fundamental to later computations.
In Section \ref{sec: spectral sequence}, we discuss several spectral sequences we use in this paper, including the algNSS, the CESS, and the May spectral sequence.
In Section \ref{sec: fourth line}, we compute relevant algebraic Novikov differentials and prove Theorem \ref{thm: new adams d2 fourth line}. 
In Section \ref{sec: first three lines}, we use the same computational strategy to recover the secondary Adams differentials on the first three lines.

\subsection*{Acknowledgments}
We would like to thank the anonymous referee for the detailed suggestions.  The third named author would also like to thank Zhilei Zhang for helpful discussions.  All authors contribute equally.

\section{Hopf algebroids} 
\label{sec: hopf algebroid}

In this section, we review the definition as well as two important examples of Hopf algebroids.  We will also recall the associated cobar complex construction.

\begin{defn}[\cite{Ravenel_stable_homotopy_groups} Definition A1.1.1]
A \textit{Hopf algebroid} over a commutative ring $K$ is a pair $(A, \Gamma)$ of commutative $K$-algebras with the following structure maps
\begin{align*}
\text{left unit map } & \eta_{L} :A\to \Gamma \\
\text{right unit map } & \eta_{R}  :A\to \Gamma \\
\text{coproduct map } & \Delta  :\Gamma \to \Gamma \otimes _{A}\Gamma  \\
\text{counit map } & \varepsilon  :\Gamma \to A \\
\text{conjugation map } & c  :\Gamma \to \Gamma 
\end{align*}
such that for any other commutative $K$-algebra $B$, the two sets of $K$-homomorphisms $\text{Hom}_K (A, B)$ and $\text{Hom}_K (\Gamma, B)$ are the objects and morphisms of a groupoid.
\end{defn}

\subsection{The Hopf algebroid $(BP_*, BP_*BP)$}

An important example of Hopf algebroids is $(BP_*, BP_*BP)$ \cite{Hazewinkel,Miller_Ravenel_Wilson,Ravenel_stable_homotopy_groups}.
Recall that we have 
\begin{equation}
\label{eq: BP generators}
    BP_*:= \pi_*(BP) = \mathbb{Z}_{(p)}[v_1, v_2, \cdots], \quad  BP_*BP = BP_* [t_1,t_2,\cdots]
\end{equation}
We also have 
\begin{equation}
    H_*(BP) = \mathbb{Z}_{(p)}[m_1, m_2, \cdots]
\end{equation}
where $|v_n| = |t_n| = |m_n| = 2(p^n - 1)$. 

\begin{notation}
Throughout this paper, we denote $v_0 = p$, and $m_0 = t_0 = 1$. 
\end{notation}

The Hurewicz map induces an embedding
\begin{equation}
\label{eq: compare vn mn}
    \begin{split}
        i: BP_* & \to H_*(BP)  \\
        v_n & \mapsto p m_n - \sum_{i = 1}^{n-1} v_{n-i}^{p^i} m_i \\
    \end{split}    
\end{equation}
We can describe the structure maps of the Hopf algebroid $(BP_*, BP_*BP)$ as follows.

The left unit and right unit maps $\eta_L, \eta_R: BP_* \to BP_*BP$ are determined by
\begin{equation}
\label{eq: BP left unit}
    \eta_L(v_n) = v_n
\end{equation}
\begin{equation}
\label{eq: BP right unit wrt m}
    \eta_R(m_n) = \sum_{i+j=n} m_i t_j^{p^i}
\end{equation}

The coproduct map $\Delta: BP_*BP \to BP_*BP \otimes_{BP_*} BP_*BP$ is determined by 
\begin{equation}
\label{eq: BP coprod wrt m}
    \sum_{i+j=n} m_i(\Delta t_j)^{p^i} = \sum_{i+j+k=n} m_i t_j^{p^i} \otimes t_k^{p^{i+j}}
\end{equation}

The counit map $\varepsilon: BP_*BP \to BP_*$ is determined by 
\begin{equation}
    \varepsilon(v_n) = v_n, \quad \varepsilon(t_n) = 0.
\end{equation}

The conjugation map $c: BP_*BP \to BP_*BP$ is determined by 
\begin{equation}
    \sum_{i+j+k=n} m_i t_j^{p^i} c(t_k)^{p^{i+j}} = m_n.
\end{equation}

In practice, it is more convenient to work with $\eta_R(v_n)$ instead of $\eta_R(m_n)$. 

Let $I$ denote the ideal $(p, v_1, v_2, \cdots) \subset BP_*$.  Then $I$ is an invariant ideal as a $BP_* BP$-comodule, in other words, we have $\eta_L(I) \cdot BP_* BP = BP_* BP \cdot \eta_R(I)$.  For $k \geq 0$, we let $I^k \cdot BP_* BP$ denote $\eta_L(I^k) \cdot BP_* BP = BP_* BP \cdot \eta_R(I^k)$.

We have the following formulas.

\begin{prop}
\label{prop: BP right unit wrt v}
Let $n \geq 0$. The right unit map $\eta_R: BP_* \to BP_*BP$ satisfies
\begin{equation}
\label{eq: BP right unit wrt v}
    \eta_R(v_n) \equiv \sum_{i=0}^n v_i t_{n-i}^{p^i}  \quad \text{mod} ~ I^p \cdot BP_* BP
\end{equation}
\end{prop}

\begin{proof}
We prove by induction on $n$.  The case for $n = 0$ is trivial.  Now suppose \eqref{eq: BP right unit wrt v} is true for $0 \leq i \leq n-1$.  Then, in particular, $\eta_R(v_i) \in I \cdot BP_* BP$ for $0 \leq i \leq n-1$. Note \eqref{eq: compare vn mn} implies 
\begin{equation}
\label{eq: vn equiv pmn}
    v_n \equiv p m_n \quad \text{mod} ~ I^p H_*(BP)
\end{equation}
for $n \geq 0$.  
Then, direct computation shows
\begin{equation}
    \begin{split}
        \eta_R(v_n) =  & ~ p ~ \eta_R(m_n) - \sum_{i = 1}^{n-1} \eta_R(v_{n-i})^{p^i} \eta_R(m_i) \quad (\text{by} ~ \eqref{eq: compare vn mn}) \\
        \equiv & ~ p ~ \eta_R(m_n) \quad \text{mod} ~ I^p \cdot BP_* BP \\
        \equiv & ~ p \sum_{i=0}^n m_i t_{n-i}^{p^i} \quad \text{mod} ~ I^p \cdot BP_* BP \quad (\text{by} ~ \eqref{eq: BP right unit wrt m})\\
        \equiv & ~ \sum_{i=0}^n v_i t_{n-i}^{p^i} \quad \text{mod} ~ I^p \cdot BP_* BP \quad (\text{by} ~ \eqref{eq: vn equiv pmn})\\
    \end{split}    
\end{equation}
\end{proof}

Similarly, we could obtain the following formulas for $\Delta(t_n)$.

\begin{prop}
\label{prop: BP coprod wrt v}
For $n \geq 0$, we have
\begin{equation}
\label{eq: BP coprod wrt v}
    \Delta(t_n) = \sum_{k=0}^n t_{n-k} \otimes t_{k}^{p^{n-k}} - \sum_{i=1}^{n-1} v_i b_{n-i, i-1}  \quad \text{mod} ~ I^2  \cdot BP_* BP  \otimes_{BP_*} BP_*BP
\end{equation}
where we denote
\begin{equation}
\label{eq: BP defn bij}
    b_{i,j} = \frac{1}{p} [ (\sum_{k=0}^i t_{i-k} \otimes t_{k}^{p^{i-k}})^{p^{j+1}} - \sum_{k=0}^i t_{i-k}^{p^{j+1}} \otimes t_{k}^{p^{i-k+j+1}}]
\end{equation}
for $i \geq 1$, $j \ge 0$.
\end{prop}

\begin{proof}
We prove by induction on $n$.  The case for $n = 0$ is trivial.  Now suppose \eqref{eq: BP coprod wrt v} is true for $0 \leq i \leq n-1$. Then, direct computation shows 
\begin{equation}
    \begin{split}
        \Delta(t_n) = & ~ \sum_{i+j=n} m_i(\Delta t_j)^{p^i} - \sum_{i=1}^{n} m_i(\Delta t_{n-i})^{p^i} \\
        = & ~ \sum_{i+j+k=n} m_i t_j^{p^i} \otimes t_k^{p^{i+j}} - \sum_{i=1}^{n} m_i(\Delta t_{n-i})^{p^i} \\
        = & ~ \sum_{k=0}^n t_{n-k} \otimes t_{k}^{p^{n-k}} + \sum_{i=1}^{n} m_i ( \sum_{k=0}^{n-i} t_{n-i-k}^{p^i} \otimes t_{k}^{p^{n-k}} ) - \sum_{i=1}^{n} m_i(\Delta t_{n-i})^{p^i} \\
        = & ~ \sum_{k=0}^n t_{n-k} \otimes t_{k}^{p^{n-k}} - \sum_{i=1}^{n} m_i [(\Delta t_{n-i})^{p^i} - \sum_{k=0}^{n-i} t_{n-i-k}^{p^i} \otimes t_{k}^{p^{n-k}}] \\
    \end{split}    
\end{equation}
Modulo $I^2 \cdot BP_* BP  \otimes_{BP_*} BP_*BP$, we have
\begin{equation*}
    \begin{split}
        & ~ \sum_{i=1}^{n} m_i [(\Delta t_{n-i})^{p^i} - \sum_{k=0}^{n-i} t_{n-i-k}^{p^i} \otimes t_{k}^{p^{n-k}}] \\
        \equiv & ~  \sum_{i=1}^{n} m_i [(\sum_{k=0}^{n-i} t_{n-i-k} \otimes t_{k}^{p^{n-i-k}})^{p^i} - \sum_{k=0}^{n-i} t_{n-i-k}^{p^i} \otimes t_{k}^{p^{n-k}}] \\
        \equiv & ~  \sum_{i=1}^{n-1} p m_i \cdot \frac{1}{p} [(\sum_{k=0}^{n-i} t_{n-i-k} \otimes t_{k}^{p^{n-i-k}})^{p^i} - \sum_{k=0}^{n-i} t_{n-i-k}^{p^i} \otimes t_{k}^{p^{n-k}}] \\
        \equiv & ~  \sum_{i=1}^{n-1} v_i b_{n-i, i-1} \\
    \end{split}    
\end{equation*}
This completes the proof.

\end{proof}

\subsection{The dual Steenrod algebra $\mathcal{A}_*$}

The Steenrod algebra provides another important example of Hopf algebroids.

Let $\mathcal{A}_*$ denote the dual mod $p$ Steenrod algebra for an odd prime $p$, we have \cite{Milnor_dual_Steenrod}
\begin{equation}
\label{eq: dual steenrod algebra milnor}
    \mathcal{A}_* = P [\xi_1, \xi_2, \cdots] \otimes E[\tau_0, \tau_1, \tau_2, \cdots]
\end{equation}
as an algebra, where $P[\xi_1, \xi_2, \cdots]$ is a polynomial algebra with coefficients in $\mathbb{F}_p$,  $E[\tau_0, \tau_1, \tau_2, \cdots]$ is an exterior algebra with coefficients in $\mathbb{F}_p$.
For the internal degrees, we have $|\xi_n| = 2(p^n - 1), |\tau_n| = 2p^n - 1$.  We also denote $\xi_0 = 1$.

One can show $\mathcal{A}_*$ is a Hopf algebra over $\mathbb{F}_p$.  In particular, $(\mathbb{F}_p, \mathcal{A}_*)$ has a Hopf algebroid structure.  We can describe the structure maps as follows \cite{Milnor_dual_Steenrod}.

The left unit $\eta_{L}: \mathbb{F}_p \to \mathcal{A}_*$, right unit $\eta_{R}: \mathbb{F}_p \to \mathcal{A}_*$, and counit $\epsilon: \mathcal{A}_* \to \mathbb{F}_p$ maps are all isomorphisms in dimension 0.

On generators, the coproduct $\Delta: \mathcal{A}_* \rightarrow \mathcal{A}_* \otimes \mathcal{A}_*$ is given by:
\begin{equation}
\label{eq: dual steenrod coproduct milnor}
    \Delta \xi_n = \sum_{i=0}^{n} \xi^{p^i}_{n-i} \otimes \xi_i, 
\quad
    \Delta \tau_n = \tau_n \otimes 1 + \sum_{i=0}^{n} \xi^{p^i}_{n-i} \otimes \tau_i
\end{equation}

The conjugation map $c: \mathcal{A}_* \rightarrow \mathcal{A}_*$ is an algebra map given recursively by 
\begin{align}
\label{eq: conjugate of steedrod first part}
   & c(\xi_0) = 1, \quad \sum_{i=0}^{n} \xi^{p^i}_{n-i} c(\xi_i) = 0, n > 0,\\
   & \quad  \tau_n + \sum_{i=0}^{n} \xi^{p^i}_{n-i} c(\tau_i) = 0, n \geq 0.
\end{align}


For our computational purposes, we prefer to use a different set of generators.  We denote $t_n = c(\xi_n), n \geq 1,$ and $\Tilde{\tau}_n = c(\tau_n), n \geq 0$.  We also denote $t_0 = 1$.

\begin{prop}
\label{prop: dual steedrod algebra new formula}
Let $p$ be an odd prime, we can write 
\begin{equation}
\label{eq: dual steenrod algebra}
    \mathcal{A}_{*} = P [t_1, t_2, \cdots] \otimes E [\Tilde{\tau}_0, \Tilde{\tau}_1, \Tilde{\tau}_2, \cdots]
\end{equation}
as an algebra, where $|t_n| = 2(p^n - 1), |\Tilde{\tau}_n| = 2p^n - 1$. Moreover, the coproduct $\Delta: \mathcal{A}_* \rightarrow \mathcal{A}_* \otimes \mathcal{A}_*$ is given by:
\begin{equation}
\label{eq: steenrod coproduct map t}
    \Delta t_n = \sum_{i=0}^{n}  t_i \otimes t^{p^i}_{n-i}, 
\quad
    \Delta \Tilde{\tau}_n = \sum_{i=0}^{n} \Tilde{\tau}_i \otimes t^{p^i}_{n-i}   + 1 \otimes \Tilde{\tau}_n
\end{equation}

\end{prop}

\begin{proof}
It is straightforward to deduce the coproduct formulas by induction on $n$.  Here, we outline the strategy to prove \eqref{eq: steenrod coproduct map t} for $t_n$. The formula for $\Tilde{\tau}_n$ can be verified similarly.

The case for $n = 0$ is trivial. Now, suppose \eqref{eq: steenrod coproduct map t} is true for $0 \leq m \leq n - 1$. Note \eqref{eq: conjugate of steedrod first part} implies 
$$\sum_{i=0}^{n-1} (\Delta \xi_{n-i})^{p^i} (\Delta t_i) + \Delta t_n = 0$$
To deduce the desired result, it suffices to show $$\sum_{i=0}^{n-1} (\Delta \xi_{n-i})^{p^i} (\Delta t_i) + \sum_{i=0}^{n}  t_i \otimes t^{p^i}_{n-i} = 0$$
Indeed, we have
\begin{equation}
    \begin{split}
         & \sum_{i=0}^{n-1} (\Delta \xi_{n-i})^{p^i} (\Delta t_i) + \sum_{i=0}^{n}  t_i \otimes t^{p^i}_{n-i} \\
         = & ~ \sum_{i=0}^{n} [(\sum_{j=0}^{n-i} \xi^{p^j}_{n-i-j} \otimes \xi_j)^{p^i} (\sum_{k=0}^{i}  t_k \otimes t^{p^k}_{i-k})] \\
         = & ~ \sum_{i=0}^{n} [(\sum_{j=0}^{n-i} \xi^{p^{i+j}}_{n-i-j} \otimes \xi_j^{p^i}) (\sum_{k=0}^{i}  t_k \otimes t^{p^k}_{i-k})] \\
         = & ~ \sum_{j+r+k+s=n} \xi^{p^{n-r}}_{r} t_k  \otimes \xi_j^{p^{k+s}} t^{p^k}_{s}\\
         = & ~ \sum_{r+k<n}  \xi^{p^{n-r}}_{r} t_k  \otimes (\sum_{j+s=n-k-r}  \xi_j^{p^{s}} t_{s} )^{p^k} + \sum_{r+k=n}  \xi^{p^{n-r}}_{r} t_k  \otimes 1\\
         = & 0  \\
    \end{split}    
\end{equation}

\end{proof}

\begin{rem}
The advantage of using the new set of generators is that, as we will see in Section \ref{sec: spectral sequence}, $c(\xi_n)$ corresponds to the generator $t_n \in BP_*BP$ and $c(\tau_n)$ corresponds to $v_n \in BP_*$. Hence, we abuse the notation and denote $c(\xi_n)$ as $t_n$ when no confusion arises.
\end{rem}

\subsection{Cobar complexes}

\begin{defn}
Let $(A, \Gamma)$ be a Hopf algebroid.  A \textit{right $\Gamma$-comodule} $M$ is a right $A$-module $M$ together with a right $A$-linear map $\psi: M \to M \otimes_A \Gamma$ which is counitary and coassociative, i.e., the following diagrams commute.
\begin{align*}
\xymatrix{
M \ar[r]^{\psi \quad} \ar@{=}[dr] & M \otimes_A \Gamma \ar[d]^{M \otimes \varepsilon} \\ 
 & M
}
\quad
\xymatrix{
M \ar[r]^{\psi \quad} \ar[d]_{\psi} & M \otimes_A \Gamma \ar[d]^{M \otimes \Delta} \\ 
M \otimes_A \Gamma \ar[r]_{\psi \otimes \Gamma \quad} & M \otimes_A \Gamma \otimes_A \Gamma
}
\end{align*}
Left $\Gamma$-comodules are defined similarly.
\end{defn}

\begin{defn}
\label{defn: cobar complex}
Let $(A, \Gamma)$ be a Hopf algebroid.  Let $M$ be a right $\Gamma$-comodule.  The cobar complex $\Omega_{\Gamma}^{*,*}(M)$ is a cochain complex with
\begin{equation*}
    \Omega_{\Gamma}^{s,*}(M) = M \otimes_A \overline{\Gamma}^{\otimes s}
\end{equation*}
where $\overline{\Gamma}$ is the augmentation ideal of $\varepsilon: \Gamma \to A$. The differentials $d: \Omega_{\Gamma}^{s,*}(M) \to \Omega_{\Gamma}^{s+1,*}(M)$ are given by
\begin{equation*}
    \begin{split}
        & d(m \otimes x_1 \otimes x_2 \otimes \cdots \otimes x_s)  = - (\psi (m) - m \otimes 1) \otimes  x_1 \otimes x_2 \otimes \cdots \otimes x_s \\
         & - \sum_{i=1}^{s} (-1)^{\lambda_{i,j_i}} m \otimes x_1 \otimes \cdots \otimes x_{i-1} \otimes (\sum_{j_i} x_{i,j_i}' \otimes x_{i,j_i}'') \otimes x_{i+1} \otimes \cdots \otimes x_s \\
    \end{split}    
\end{equation*}
where
$$\sum_{j_i} x_{i,j_i}' \otimes x_{i,j_i}'' = \Delta (x_i) - 1 \otimes x_i - x_i \otimes 1$$
$$\lambda_{i,j_i} = i + |x_1|  + \cdots + |x_{i-1}| + |x_{i,j_i}'|$$

\end{defn}

The cohomology of $\Omega_{\Gamma}^{s,*}(M)$ is $Ext_{\Gamma}^{s,*}(A, M)$ (see \cite[Section A1.2]{Ravenel_stable_homotopy_groups}).

\section{Some relevant spectral sequences}
\label{sec: spectral sequence}

In this section, we review the construction and properties of some relevant spectral sequences, including the algebraic Novikov spectral sequence (algNSS), the Cartan-Eilenberg spectral sequence (CESS), and the May spectral sequence (MSS). These spectral sequences will be used in later computations.

\subsection{The algebraic Novikov spectral sequence}

Let $I$ be the ideal of $BP_*$ generated by $(p, v_1, v_2, \cdots)$. 
The ideal $I$ induces a filtration 
\begin{equation}
    BP_* = I^0 \supset I^1 \supset I^2 \supset I^3 \supset \cdots \supset I^k \supset I^{k+1} \supset \cdots
\end{equation}
Consider $y = a p^{k_0} v_1^{k_1} v_2^{k_2} \cdots \in BP_*$, where $a \in \mathbb{Z}_{(p)}$ is invertible.  We let $l(y) = \Sigma_i k_i$ denote the length of $y$.
Then $y \in I^k$ if and only if $l(y) \geq k$. 

Let $E_0^* BP_*$ denote the associated graded object, where $E_0^k BP_*:= I^k/I^{k+1}$. We have
\begin{equation}
    E_0^* BP_* = \bigoplus_{k \geq 0} I^k/I^{k+1} = \mathbb{F}_{p}[q_0, q_1, q_2, \cdots]
\end{equation}
is a $\mathbb{F}_{p}$-coefficient polynomial algebra, where the generator $q_i$ corresponds to $v_i$, $I^k/I^{k+1}$ corresponds to those homogeneous polynomials of degree $k$. 

Similarly, we can filter $BP_* BP$.  Denote 
$$F^k BP_* BP:= \eta_L(I^k) BP_* BP = BP_* BP \eta_R(I^k)$$
We define the associated graded object $E_0^k BP_*BP:= F^k BP_* BP/F^{k+1} BP_* BP$.  The filtration of $BP_*$ and $BP_* BP$ together induces a filtration on $\Omega_{BP_* BP} (BP_*)$.  Such filtration induces an associated spectral sequence \cite[A1.3.9]{Ravenel_stable_homotopy_groups} converging to $Ext^{s,t}_{BP_* BP} (BP_*, BP_*)$.

\begin{thm}[\cite{Miller_algANSS,Novikov_methods}]
There is a spectral sequence, called the algebraic Novikov spectral sequence (algNSS), converging to $Ext^{s,t}_{BP_* BP} (BP_*, BP_*)$ with $E_2$-page
$$E_2^{s,t,k} = Ext^{s,t}_{P_*} (\mathbb{F}_{p}, I^k / I^{k+1})$$ 
and $d_r^{alg}: E_r^{s,t,k} \to E_r^{s+1,t,k+r-1}$, where
\begin{equation}
    P_* := E_0 BP_*BP \otimes_{E_0 BP_*} \mathbb{F}_{p} = BP_* BP/I = P[t_1, t_2, \cdots]
\end{equation}
is the $\mathbb{F}_{p}$-coefficient polynomial algebra.

\end{thm}

\begin{rem}
Our index of pages here is different from the ones used in \cite{Andrews_Miller, Ravenel_stable_homotopy_groups}.  We have re-indexed the spectral sequence to align with the notations in \cite{Gheorghe_Wang_Xu, Isaksen_Wang_Xu_more_stable_stems}.
\end{rem}

\subsection{The Cartan-Eilenberg spectral sequence}

Let $\mathcal{A}_{*}$ denote the dual Steenrod algebra for an odd prime $p$.  Recall from Proposition \ref{prop: dual steedrod algebra new formula} that we have 
$$\mathcal{A}_{*} = P [t_1, t_2, \cdots] \otimes E [\Tilde{\tau}_0, \Tilde{\tau}_1, \Tilde{\tau}_2, \cdots]$$

Let $P_*$ denote $P [t_1, t_2, \cdots] \subset \mathcal{A}_{*}$.
Let $E_*$ denote $E [\Tilde{\tau}_0, \Tilde{\tau}_1, \Tilde{\tau}_2, \cdots]$. Then 
\begin{equation*}
    P_* \to \mathcal{A}_{*} \to E_*
\end{equation*}
is an extension of Hopf algebras \cite[A1.1.15]{Ravenel_stable_homotopy_groups}, which induces a spectral sequence \cite[A1.3.14]{Ravenel_stable_homotopy_groups} converging to $Ext_{\mathcal{A}_{*}}^{*,*} (\mathbb{F}_p, \mathbb{F}_p)$.

\begin{thm}[\cite{Ravenel_stable_homotopy_groups} Theorem 4.4.3, 4.4.4]
\label{thm: CESS props}
Let $p$ be an odd prime.  There is a spectral sequence, called the Cartan-Eilenberg spectral sequence (CESS), converging to $Ext^{s_1+s_2,t}_{\mathcal{A}_{*}} (\mathbb{F}_p, \mathbb{F}_p)$ with $E_2$-page
$$E_2^{s_1,t,s_2} = Ext_{P_*}^{s_1,t} (\mathbb{F}_p, Ext_{E_*}^{s_2} (\mathbb{F}_p, \mathbb{F}_p))$$ 
and $d_r: E_r^{s_1,t,s_2} \to E_r^{s_1+r,t,s_2-r+1}$.  Moreover, one can prove the following results:

(a) $Ext_{E_*}^{s_2, *} (\mathbb{F}_p, \mathbb{F}_p) = P[a_0, a_1, \cdots]$ is a polynomial algebra with generator $a_i \in Ext^{1, 2p^i-1}$ represented in the associated cobar complex $\Omega_{E_*}(\mathbb{F}_p)$ by $[\Tilde{\tau}_i]$.

(b) The $P_*$-coaction on $Ext_{E_*} (\mathbb{F}_p, \mathbb{F}_p)$ is given by
\begin{equation}
\label{eq: structure map of an in CESS}
    \psi(a_n)  = \sum_{i=0}^n a_i \otimes t_{n-i}^{p^i}  
\end{equation}

(c) The CESS collapses from $E_2$ with no nontrivial extensions.

(d) There is an isomorphism 
\begin{equation}
\label{eq: iso of E2}
    Ext_{P_*}^{s,t} (\mathbb{F}_{p}, I^k / I^{k+1}) \cong Ext_{P_*}^{s,t+k} (\mathbb{F}_p, Ext_{E_*}^{k} (\mathbb{F}_p, \mathbb{F}_p))
\end{equation}
between the $E_2$-page of the algNSS and the $E_2$-page of the CESS.
\end{thm}

The (d) part shows the two $Ext$ groups are isomorphic (up to degree shifting). Moreover, we can show the two associated cobar complexes are isomorphic (up to a shifting of degrees). More precisely, there is a natural isomorphism 
\begin{equation}
    \Omega_{P_*} (I^k / I^{k+1}) \cong \Omega_{P_*} (Ext_{E_*}^{k} (\mathbb{F}_p, \mathbb{F}_p))
\end{equation}
sending $t_i$ to $t_i$ and $q_i$ to $a_i$.

Indeed, by Theorem \ref{thm: CESS props} (a), $Ext_{E_*}^{k} (\mathbb{F}_p, \mathbb{F}_p)$ is the homogeneous degree $k$ part of the polynomial $P[a_0, a_1, a_2, \cdots]$.  Hence $I^k / I^{k+1} \cong Ext_{E_*}^{k} (\mathbb{F}_p, \mathbb{F}_p)$. If $x \in I^k / I^{k+1}$ has inner degree $t$, then its corresponding element $\Tilde{x} \in Ext_{E_*}^{k} (\mathbb{F}_p, \mathbb{F}_p)$ has inner degree $t+k$.  This degree shifting is a consequence of the fact that $|q_i| = 2(p^i-1) = |a_i| - 1$.  Moreover, the comodule structure map $\psi: I^k/I^{k+1} \to I^k/I^{k+1} \otimes P_*$ induced from \eqref{eq: BP right unit wrt v} is given by
\begin{equation}
\label{eq: psi map for qn in alg}
    \psi(q_n)  = \sum_{i=0}^n q_i \otimes t_{n-i}^{p^i}  
\end{equation}
which also agrees with \eqref{eq: structure map of an in CESS}.

\begin{notation}
\label{notation: denote E2 element by cobar}
In this paper, we often refer to $E_2$-terms of the algNSS by their representative in the cobar complex $\Omega_{P_*} (I^k / I^{k+1})$. For example, we let $q_0 \otimes t_1^p$ denote its homology class in $Ext_{P_*}^{1,*} (\mathbb{F}_{p}, I / I^{2})$. The correspondence between different $E_2$-pages becomes clear under this naming convention.  For example, $q_0 \otimes t_1^p \in Ext_{P_*}^{1,*} (\mathbb{F}_{p}, I / I^{2})$ in the algNSS corresponds to $a_0 \otimes t_1^p \in Ext_{P_*}^{1,*} (\mathbb{F}_{p}, Ext_{E_*}^{1} (\mathbb{F}_p, \mathbb{F}_p))$ in the CESS, which represents $\Tilde{\tau}_0 \otimes t_1^p \in Ext_{\mathcal{A}_{*}}^{2,*} (\mathbb{F}_{p}, \mathbb{F}_{p})$ in the ASS.

\end{notation}

\subsection{The May spectral sequence}

The $E_2$-terms $Ext_{\mathcal{A}_{*}}^{s,t} (\mathbb{F}_p, \mathbb{F}_p)$ of the Adams spectral sequence could be computed via the cobar complex $\Omega_{\mathcal{A}_*}^{*,*}(\mathbb{F}_p)$.  In practice, we could simplify such computations by filtering $\Omega_{\mathcal{A}_*}^{*,*}(\mathbb{F}_p)$. 

\begin{thm}[\cite{May_cohomology_Lie_algebra}, \cite{Ravenel_stable_homotopy_groups} Theorem 3.2.5]
\label{thm: MSS}
Let $p$ be an odd prime.  $\mathcal{A}_{*}$ can be given
an increasing filtration by setting the May degree $M(t_i^{p^j}) = M(\Tilde{\tau}_{i-1}) = 2i-1$ for $i-1, j \geq 0$.  The filtration of $\mathcal{A}_{*}$ naturally induces a filtration of $\Omega_{\mathcal{A}_*}^{*,*}(\mathbb{F}_p)$.  The associated spectral sequence converging to $Ext_{\mathcal{A}_{*}}^{s,t} (\mathbb{F}_p, \mathbb{F}_p)$ is called the May spectral sequence (MSS).  The MSS has $E_1$ page
\begin{equation}
    E_1^{*,*,*} = E [h_{i,j} | i \geq 1, j \geq 0] \otimes P [b_{i,j} | i \geq 1, j \geq 0] \otimes P [a_i |i \geq 0]
\end{equation}
and $d_r: E_r^{s, t, M} \to E_r^{s+1, t, M-r}$, where
\begin{equation}
    \begin{split}
        h_{i,j} = & [t_{i}^{p^j}] \in E^{1,2(p^i-1)p^j,2i-1}_1 \\
        b_{i,j} = & [ ~ \sum_{k=1}^{p-1}  \binom{p}{k}/p ~ (t_{i}^{p^j})^k \otimes (t_{i}^{p^j})^{p-k} ~ ] \in E^{2,2(p^i-1)p^{j+1},p(2i-1)}_1 \\
        a_{i} = & [\Tilde{\tau}_{i}] \in E^{1,2p^i-1,2i+1}_1 \\
    \end{split}    
\end{equation}
\end{thm}

\begin{rem}
\label{rem: two bij}
Technically, we could denote the generator by $\Tilde{b}_{i,j}$ instead of $b_{i,j}$ to avoid possible confusion with the element 
$$b_{i,j} = \frac{1}{p} [ (\sum_{k=0}^i t_{i-k} \otimes t_{k}^{p^{i-k}})^{p^{j+1}} - \sum_{k=0}^i t_{i-k}^{p^{j+1}} \otimes t_{k}^{p^{i-k+j+1}}]$$
defined in \eqref{eq: BP defn bij}. However, let $x$ be the element in $\Omega_{\mathcal{A}_*}^{*,*}(\mathbb{F}_p)$ corresponding to $b_{i,j}$ (Notations \ref{notation: denote E2 element by cobar}). Note $\Omega_{\mathcal{A}_*}^{*,*}(\mathbb{F}_p)$ has coefficient  $\mathbb{F}_p$, we have 
\begin{equation*}
    \begin{split}
        x & = \frac{1}{p} \sum_{k_1 \neq k_2} \sum_{t=1}^{p-1} \binom{p^{j+1}}{t p^j} (t_{i-k_1} \otimes t_{k_1}^{p^{i-k_1}})^{t p^j} (t_{i-k_2} \otimes t_{k_2}^{p^{i-k_2}})^{(p-t) p^j} \\
        & = \frac{1}{p} \sum_{k_1 \neq k_2} \sum_{t=1}^{p-1} \binom{p}{t} (t_{i-k_1}^{t p^j} t_{i-k_2}^{(p-t) p^j} \otimes t_{k_1}^{tp^{i+j-k_1}} t_{k_2}^{(p-t)p^{i+j-k_2}}) \\
    \end{split}    
\end{equation*}
Its May filtration leading term is 
$$\frac{1}{p} \sum_{t=1}^{p-1} \binom{p}{t} t_i^{t p^j} \otimes t_i^{(p-t) p^j} = \Tilde{b}_{i,j}$$
Therefore, we often abuse the notation and also denote $\Tilde{b}_{i,j}$ by $b_{i,j}$.
\end{rem}

Note we can analogously define an increasing filtration on $\Omega_{P_*} (I^k / I^{k+1})$ (hence also on $\Omega_{P_*} (Ext_{E_*}^{k} (\mathbb{F}_p, \mathbb{F}_p))$) by setting the May degree $M(t_i^{p^j}) = M(q_{i-1}) = 2i-1$ for $i-1, j \geq 0$. 
We observe the following structure maps:
\begin{equation}
    \psi(q_n)  = \sum_{i=0}^n q_i \otimes t_{n-i}^{p^i}, \quad \Delta t_n = \sum_{i=0}^{n}  t_i \otimes t^{p^i}_{n-i}.  
\end{equation}
For $i < n$, we have $M(q_n) = 2n+1 \geq 2n = 2i+1 + 2(n-i)-1 = M(q_i \otimes t_{n-i}^{p^i})$.  Similarly, for $0<i<n$, $M(t_n) = 2n-1 \geq 2n -2 = M(t_i \otimes t^{p^i}_{n-i})$.  Let $d$ denote the differential of the cobar complex $\Omega_{P_*} (I^k / I^{k+1})$ (see Definition \ref{defn: cobar complex}).  Then $d$ respects this May filtration.  Hence, we can talk about the May filtration of the algNSS $E_2$-terms.  Moreover, the May filtration of the elements in the algNSS $E_2$-page agrees with the May filtration of the corresponding elements in the ASS $E_2$-page (see Notations \ref{notation: denote E2 element by cobar}).

\section{Secondary Adams differentials on the fourth line}
\label{sec: fourth line}

In this section, we prove our main result Theorem \ref{thm: new adams d2 fourth line}.  Using Theorem \ref{thm: miller transition}, we determine these secondary Adams differentials $d_2^{Adams}$ by computing their corresponding secondary  algebraic Novikov differentials $d_2^{alg}$.

Our computational strategy in this paper can be summarized as follows:
\begin{enumerate}
    \item Let $x$ be an element in the Adams $E_2$-page.  Let $l$ be the MSS representative of $x$. 
    \item As stated in Notations \ref{notation: denote E2 element by cobar}, we find the the element $x'$ (resp. $l'$) in the algebraic Novikov spectral sequence corresponding to $x$ (resp. $l$).   We deduce $l'$ is the May filtration leading term of $x'$. 
    \item Through a careful analysis of $l'$, we determine the May filtration leading term $y'$ of $d_2^{alg} (x')$.
    \item Let $y$ be the element in the MSS corresponding to $y'$. Then we conclude $d_2^{Adams} (x)$ is represented by $y$.
\end{enumerate}

In particular, we will use Table \ref{table: fourth line elements representatives} for the four families of Adams $E_2$-terms in Theorem \ref{thm: new adams d2 fourth line}.

\renewcommand{\arraystretch}{1.6}

\begin{table}[ht!]
\centering
\begin{tabular}{| c | c | c | c |} 
 \hline
 Adams $E_2$-term $x$ & MSS representative $l$ & corresponding algNSS term $l'$\\ 
 \hline
 $h_{4, i} h_{3, i} g_i$ & $h_{4,i} h_{3, i} h_{2, i} h_{1, i}$ & $t_4^{p^{i}} \otimes t_3^{p^{i}} \otimes t_2^{p^{i}} \otimes t_1^{p^{i}}$  \\
 \hline
 $h_{4, i} h_{3, i+1} k_{i+2}$ & $h_{4,i} h_{3, i+1} h_{2, i+2} h_{1, i+3}$ & $t_4^{p^{i}} \otimes t_3^{p^{i+1}} \otimes t_2^{p^{i+2}} \otimes t_1^{p^{i+3}}$ \\
 \hline
 $h_{4, i} g_i h_{i+3}$ & $h_{4,i} h_{2, i} h_{1, i} h_{1, i+3}$ & $t_4^{p^{i}} \otimes t_2^{p^{i}} \otimes t_1^{p^{i}} \otimes t_1^{p^{i+3}}$ \\
 \hline 
 $h_{3, i} h_{2, i+1} k_{i}$ & $h_{3,i} h_{2, i+1} h_{2, i} h_{1, i+1}$ & $t_3^{p^{i}} \otimes t_2^{p^{i+1}} \otimes t_2^{p^{i}} \otimes t_1^{p^{i+1}}$  \\
 \hline  
\end{tabular}
\caption{Representations of the four elements}
\label{table: fourth line elements representatives}
\end{table}

Now we start the actual computations.

\begin{lem}
\label{lem: compute d(tkpi)}
Let $d$ denote the differential in the cobar complex $\Omega_{BP_* BP}^{*,*}(BP_*)$ (Definition \ref{defn: cobar complex}). Let $n,i \geq 1$, we have
\begin{equation}
    d(t_n^{p^i}) = \sum_{k=1}^{n-1} t_{n-k}^{p^i} \otimes t_{k}^{p^{n-k+i}} + p b_{n,i-1} \quad \text{mod} ~ I^2 \cdot BP_* BP  \otimes_{BP_*} BP_*BP
\end{equation}
\end{lem}

\begin{proof}

After reduction module $I^2 \cdot BP_* BP  \otimes_{BP_*} BP_*BP$, we have 
\begin{equation*}
    \begin{split}
        d(t_n^{p^i}) & = \Delta (t_n^{p^i}) - 1 \otimes t_n^{p^i} - t_n^{p^i} \otimes 1 \\
        & = (\sum_{k=0}^n t_{n-k} \otimes t_{k}^{p^{n-k}} - \sum_{i=1}^{n-1} v_i b_{n-i, i-1})^{p^i} - 1 \otimes t_n^{p^i} - t_n^{p^i} \otimes 1 \quad (\text{by} ~ \eqref{eq: BP coprod wrt v}) \\
        & = (\sum_{k=0}^n t_{n-k} \otimes t_{k}^{p^{n-k}})^{p^i} - 1 \otimes t_n^{p^i} - t_n^{p^i} \otimes 1 \\
        & = \sum_{k=1}^{n-1} t_{n-k}^{p^i} \otimes t_{k}^{p^{n-k+i}} + p b_{n,i-1} \quad (\text{compare with} ~ \eqref{eq: BP defn bij})\\
    \end{split}    
\end{equation*}

\end{proof}

\begin{prop}
\label{prop: compute alg d2 first example}
Let $x \in Ext^{4,*}_{P_*} (\mathbb{F}_{p}, BP_* / I)$ be an element in the $E_2$-page of the algNSS such that $x$ has May filtration leading term $t_4^{p^{i}} \otimes t_3^{p^{i}} \otimes t_2^{p^{i}} \otimes t_1^{p^{i}}$, where $i \geq 1$. Then $d_2^{alg} (x)$ has May filtration leading term $q_0 b_{4, i-1} \otimes t_3^{p^{i}} \otimes t_2^{p^{i}} \otimes t_1^{p^{i}}$.
\end{prop}

\begin{proof}

We will compute $d_2^{alg}(x)$ as follows. First, we will find a representative $\Tilde{x}$ of $x$ in $\Omega_{BP_* BP}^{4,*}(BP_*)$. Afterward, we will analyze $d(\Tilde{x})$, where $d: \Omega_{BP_* BP}^{4,*}(BP_*) \to \Omega_{BP_* BP}^{5,*}(BP_*)$ denotes the differential in the cobar complex $\Omega_{BP_* BP}^{*,*}(BP_*)$.  This analysis will provide us with the necessary information about $d(\Tilde{x})$, which represents $d_2^{alg}(x) \in Ext^{5,*}_{P_*}(\mathbb{F}_{p}, I / I^2)$.

Using Lemma \ref{lem: compute d(tkpi)} and the Leibniz rule, we have 
\begin{equation}
    \begin{split}
        & d(t_4^{p^{i}} \otimes t_3^{p^{i}} \otimes t_2^{p^{i}} \otimes t_1^{p^{i}})
         = d(t_4^{p^{i}})  \otimes t_3^{p^{i}} \otimes t_2^{p^{i}} \otimes t_1^{p^{i}} - t_4^{p^{i}} \otimes d(t_3^{p^{i}}) \otimes t_2^{p^{i}} \otimes t_1^{p^{i}} \\
        & \quad \quad + t_4^{p^{i}} \otimes t_3^{p^{i}} \otimes d(t_2^{p^{i}}) \otimes t_1^{p^{i}} - t_4^{p^{i}} \otimes t_3^{p^{i}} \otimes t_2^{p^{i}} \otimes d(t_1^{p^{i}})\\
        & \quad \quad \equiv R + p b_{4, i-1} \otimes t_3^{p^{i}} \otimes t_2^{p^{i}} \otimes t_1^{p^{i}} + L \quad \text{mod} ~ I^2 \cdot BP_* BP^{\otimes 5} \\
        & \quad \quad \equiv R \quad  \text{mod} ~ I \cdot BP_* BP^{\otimes 5} \\
    \end{split}    
\end{equation}  
where we denote
\begin{equation}
    \begin{split}
        R & = (t_{3}^{p^{i}} \otimes t_1^{p^{i+3}} + t_{2}^{p^{i}} \otimes t_2^{p^{i+2}} + t_{1}^{p^{i}} \otimes t_3^{p^{i+1}})  \otimes t_3^{p^{i}} \otimes t_2^{p^{i}} \otimes t_1^{p^{i}} \\
        & - t_4^{p^{i}} \otimes (t_{2}^{p^{i}} \otimes t_1^{p^{i+2}} + t_{1}^{p^{i}} \otimes t_2^{p^{i+1}}) \otimes t_2^{p^{i}} \otimes t_1^{p^{i}} + t_4^{p^{i}} \otimes t_3^{p^{i}} \otimes t_1^{p^i} \otimes t_1^{p^{i+1}} \otimes t_1^{p^{i}}, \\
    \end{split}    
\end{equation}  
and 
\begin{equation}
    \begin{split}
        L & = - t_4^{p^{i}} \otimes p b_{3, i-1} \otimes t_2^{p^{i}} \otimes t_1^{p^{i}}  + t_4^{p^{i}} \otimes t_3^{p^{i}} \otimes p b_{2, i-1} \otimes t_1^{p^{i}} - t_4^{p^{i}} \otimes t_3^{p^{i}} \otimes t_2^{p^{i}} \otimes p b_{1, i-1}\\
    \end{split}    
\end{equation}  
which is a sum of monomials in $I \cdot BP_* BP^{\otimes 5}$ with May degrees lower than $M(p b_{4, i-1} \otimes t_3^{p^{i}} \otimes t_2^{p^{i}} \otimes t_1^{p^{i}}) = 7p + 10$.

Since $x \in Ext^{4,*}_{P_*} (\mathbb{F}_{p}, BP_* / I)$ has May filtration leading term $t_4^{p^{i}} \otimes t_3^{p^{i}} \otimes t_2^{p^{i}} \otimes t_1^{p^{i}}$, we can choose a representative $\Tilde{x}$ of $x$ in $\Omega_{BP_* BP}^{4,*}(BP_*) = BP_* BP^{\otimes 4}$ in the form of
\begin{equation}
    \Tilde{x} = t_4^{p^{i}} \otimes t_3^{p^{i}} \otimes t_2^{p^{i}} \otimes t_1^{p^{i}} - \sum_r y_r,
\end{equation}
such that: 
\begin{itemize}
    \item[(a)] each $y_r$ is a monomial in $BP_* BP^{\otimes 4}$ and is not an element of $I \cdot BP_* BP^{\otimes 4}$,
    \item[(b)] $M(y_r) < M(t_4^{p^{i}} \otimes t_3^{p^i} \otimes t_2^{p^i} \otimes t_1^{p^i}) = 7 + 5 + 3 + 1 = 16$,
    \item[(c)] $\sum_r d(y_r) \equiv  R ~ \text{mod} ~ I \cdot BP_* BP^{\otimes 5}$, ensuring that $d(\Tilde{x}) \equiv 0 ~ \text{mod} ~ I \cdot BP_* BP^{\otimes 5}$.
\end{itemize}

For each $r$, we express $d(y_r)$ as a sum of monomials in $BP_* BP^{\otimes 5}$:
\begin{equation}
    d(y_r) = \sum_u z_{r,u}.
\end{equation}
Next, we define the sets $A_r:= \{z_{r,u}| z_{r,u} \notin I \cdot BP_* BP^{\otimes 5} \}$ and $B_r:= \{z_{r,u}| z_{r,u} \in I \cdot BP_* BP^{\otimes 5},  z_{r,u} \notin I^2 \cdot BP_* BP^{\otimes 5} \}$, which correspond to the (possibly empty) sets of summands. Using these sets, we then obtain:
\begin{equation}
    0 \equiv d(\Tilde{x}) \equiv R - \sum_r \sum_{z_{r,u} \in A_r} z_{r,u} \quad \text{mod} ~ I \cdot BP_* BP^{\otimes 5}
\end{equation}
\begin{equation}
    d(\Tilde{x}) \equiv p b_{4, i-1} \otimes t_3^{p^{i}} \otimes t_2^{p^{i}} \otimes t_1^{p^{i}} + L - \sum_r \sum_{z_{r,u} \in B_r} z_{r,u} \quad \text{mod} ~ I^2 \cdot BP_* BP^{\otimes 5}
\end{equation}
Therefore, $p b_{4, i-1} \otimes t_3^{p^{i}} \otimes t_2^{p^{i}} \otimes t_1^{p^{i}} + L - \sum_r \sum_{z_{r,u} \in B_r} z_{r,u}$ represents $d_2^{alg} (x) \in Ext^{5,*}_{P_*}(\mathbb{F}_{p}, I / I^2)$. 

The condition $M(y_r) < 16$ strongly restricts the form of $y_r$. To show that $M(z_{r,u}) < M(p b_{4, i-1} \otimes t_3^{p^{i}} \otimes t_2^{p^{i}} \otimes t_1^{p^{i}}) = 7p + 10$ holds for all $z_{r,u} \in B_r$, we can conduct a tedious but straightforward check through all possible forms of $y_r$. Alternatively, we can summarize the idea as follows, considering three different cases:
\begin{itemize}
    \item[(a)] If $y_r = t_4^{p^{k}} \otimes A$ with $k \geq 1$, where $A$ is made up of $t_1$, $t_2$, and $t_3$ terms and $M(A) \leq 8$, then we have $M(z_{r,u}) \leq M(pb_{4,k-1}  \otimes A) = 7p + 1 + M(A) \leq 7p+9 < 7p + 10$.
    \item[(b)] If $y_r = t_4 \otimes A$, where $A$ is made up of $t_1$, $t_2$, and $t_3$ terms, and $M(A) \leq 8$, we note that $d(t_4) = t_{3} \otimes t_1^{p^3} + t_{2} \otimes t_2^{p^2} + t_{1} \otimes t_3^{p} - v_1 b_{3,0} - v_2 b_{2,1} - v_3 b_{1,2}$. Also, $M(b_{i,j}) = p(2i-1) \leq 5p$ for $i \leq 3$. We can then observe that $M(z_{r,u}) < 7p + 10$.
    \item[(c)] If $y_r$ is made up of $t_1$, $t_2$, and $t_3$ terms, we can also use similar ideas and check that $M(z_{r,u}) <  7p + 10$.
\end{itemize}
Thus, we conclude $d_2^{alg} (x)$ has May filtration leading term $q_0 b_{4, i-1} \otimes t_3^{p^{i}} \otimes t_2^{p^{i}} \otimes t_1^{p^{i}}$.
\end{proof}

We can compute the following differentials similarly to Proposition \ref{prop: compute alg d2 first example}.

\begin{prop}
\label{prop: compute alg d2 other examples}

We have the following secondary algebraic Novikov differentials.  
\begin{enumerate}
    \item $d_2^{alg}(t_4^{p^{i}} \otimes t_3^{p^{i+1}} \otimes t_2^{p^{i+2}} \otimes t_1^{p^{i+3}}) = q_0 b_{4, i-1} \otimes t_3^{p^{i+1}} \otimes
    t_2^{p^{i+2}} \otimes t_1^{p^{i+3}}$, for $i \geq 1$.
    \item $d_2^{alg}(t_4^{p^{i}} \otimes t_2^{p^{i}} \otimes t_1^{p^{i}} \otimes t_1^{p^{i+3}}) = q_0 b_{4, i-1} \otimes t_2^{p^{i}} \otimes t_1^{p^{i}} \otimes t_1^{p^{i+3}}$, for $i \geq 1$.
    \item $d_2^{alg}(t_3^{p^{i}} \otimes t_2^{p^{i+1}} \otimes t_2^{p^{i}} \otimes t_1^{p^{i+1}}) = q_0 b_{3, i-1} \otimes t_2^{p^{i+1}} \otimes t_2^{p^{i}} \otimes t_1^{p^{i+1}}$, for $i \geq 1$.
\end{enumerate}
Here, the equations hold after modding out lower May filtration terms.
\end{prop}

\begin{proof}
These results can be computed directly analogous to Proposition \ref{prop: compute alg d2 first example}.
\end{proof}

\begin{thm}
\label{thm: new adams d2 fourth line}
There are nontrivial secondary Adams differentials given as follows:
\begin{enumerate}
    \item $d_2^{Adams} (h_{4, i} h_{3, i} g_i) = a_0 b_{4,i-1} h_{3, i} g_i$, for $i \geq 1$.
    \item $d_2^{Adams} (h_{4, i} h_{3, i+1} k_{i+2}) = a_0 b_{4,i-1} h_{3, i+1} k_{i+2}$, for $i \geq 1$.
    \item $d_2^{Adams} (h_{4, i} g_i h_{i+3}) = a_0 b_{4,i-1} g_i h_{i+3}$, for $i \geq 1$.
    \item $d_2^{Adams} (h_{3, i} h_{2, i+1} k_{i}) = a_0 b_{3,i-1} h_{2, i+1} k_{i}$, for $i \geq 1$.
\end{enumerate}
\end{thm}

\begin{proof}
These results can be directly deduced from Propositions \ref{prop: compute alg d2 first example} and \ref{prop: compute alg d2 other examples}.  Moreover, these differentials are all nontrivial.   We can take $a_0 b_{4,i-1} h_{3, i} g_i$ as an example to show $a_0 b_{4,i-1} h_{3, i} g_i \neq 0 \in Ext_{\mathcal{A}_*}^{6,*}$.  The other three cases are similar.

Note $a_0 b_{4,i-1} h_{3, i} g_i$ has May spectral sequence representative 
$$a_0 b_{4, i-1} h_{3, i} h_{2, i} h_{1, i} \in E_1^{6,t,M}$$
Here the inner degree is
$$t = 1 + qp^i((1+p+p^2+p^3)+(1+p+p^2)+(1+p)+1)$$
where we denote $q = 2(p-1)$.  Let $x$ be an element in 
$E_1^{5,t,*}$.  Inspection of degrees shows $x$ must be $a_0 h_{4, i} h_{3, i} h_{2, i} h_{1, i}$.  Then $M(x) < M(a_0 b_{4, i-1} h_{3, i} h_{2, i} h_{1, i})$.  Hence $a_0 b_{4, i-1} h_{3, i} h_{2, i} h_{1, i}$ can not be the image of any May differential 
$d_r: E_r^{5, t, M+r} \to E_r^{6, t, M}$, $r \geq 1$.
This completes the proof.
\end{proof} 

It is worth pointing out that Zhong-Hong-Zhao \cite{Zhong_Hong_Zhao} also computed two other nontrivial differentials on the fourth line.

\begin{thm}[Zhong-Hong-Zhao \cite{Zhong_Hong_Zhao}]
\label{thm: 2 fourth line adams d2}
On the fourth line $Ext_{\mathcal{A}_*}^{4,*} (\mathbb{F}_p, \mathbb{F}_p)$ of the Adams spectral sequence, there exist two nontrivial secondary Adams differentials given as follows:
\begin{enumerate}
    \item $d_2^{Adams} (h_{3,i} g_i h_{2,i-1}) = a_0 b_{3,i-1} g_i h_{2,i-1}$ for $i \geq 2$.
    \item $d_2^{Adams} (h_{3,i} k_{i+1} h_{2,i+2}) = a_0 b_{3,i-1} k_{i+1} h_{2,i+2}$ for $i \geq 1$.
\end{enumerate}

\end{thm}

Their result can be recovered by computing the following corresponding algebraic Novikov differentials.

\begin{prop}
\label{prop: alg for 2 fourth line adams d2}
We have the following secondary algebraic Novikov differentials.
Here, the equations hold after modding out lower May filtration terms.
\begin{enumerate}
    \item $d_2^{alg}(t_3^{p^{i}} \otimes t_2^{p^{i}} \otimes t_1^{p^{i}} \otimes t_2^{p^{i-1}}) = q_0 b_{3, i-1} \otimes t_2^{p^{i}} \otimes t_1^{p^{i}} \otimes t_2^{p^{i-1}}$, for $i \geq 2$.
    \item $d_2^{alg}(t_3^{p^{i}} \otimes t_2^{p^{i+1}} \otimes t_1^{p^{i+2}} \otimes t_2^{p^{i+2}}) = q_0 b_{3, i-1} \otimes t_2^{p^{i+1}} \otimes t_1^{p^{i+2}} \otimes t_2^{p^{i+2}}$, for $i \geq 1$.
\end{enumerate}

\end{prop}

\begin{proof}
These results can be computed directly analogous to Proposition \ref{prop: compute alg d2 first example}.
\end{proof}

Our computations here are comparatively more straightforward than the original computations in \cite{Zhong_Hong_Zhao} using matrix Massey products.

\section{Secondary Adams differentials on the first three lines}
\label{sec: first three lines}

In this section, we use the strategy explained in Section \ref{sec: fourth line} to recover secondary Adams differentials on the first three lines.

The generators for the first two lines of the Adams spectral sequence were determined by Liulevicius in \cite{Liulevicius_factorization}.
We summarize them in the following table.

\renewcommand{\arraystretch}{1.2}

\begin{longtable}{| c | c | c | c |}
 \hline
 Generator & Representation in MSS & Inner Degree & Range of indices\\ 
 \hline
 $a_0$ & $a_0$ & 1 &  \\
 \hline
 $h_i$ & $h_{1,i}$ & $qp^i$ & $i \geq 0$ \\
 \hline
 $a_1 h_0$ & $a_1 h_{1,0}$ & $2q+1$ &  \\
 \hline 
 $a_0^2$ & $a_0^2$ & $2$ &  \\
 \hline
 $a_0 h_i$ & $a_0 h_{1,i}$ & $qp^i+1$ & $i \geq 1$ \\
 \hline
 $g_i$ & $h_{2,i} h_{1,i}$ & $q(2p^i+p^{i+1})$ & $i \geq 0$ \\
 \hline 
 $k_i$ & $h_{2,i} h_{1,i+1}$ & $q(p^i+2p^{i+1})$ & $i \geq 0$ \\
 \hline  
 $b_i$ & $b_{1,i}$ & $qp^{i+1}$ & $i \geq 0$ \\
 \hline   
 $h_i h_j$ & $h_{1,i} h_{1,j}$ & $q(p^i+p^{j})$ & $j-2 \geq i \geq 0$\\
 \hline   
\caption{\label{table: first and second line basis} A $\mathbb{F}_p$-basis of $Ext_{\mathcal{A}_*}^{1,*}$ and $Ext_{\mathcal{A}_*}^{2,*}$}
\end{longtable}

For odd primes, Aikawa \cite{Aikawa_3_dimensional} determined a basis for $Ext_{\mathcal{A}_*}^{3,*}$ using $\Lambda$-algebra.  For $p \geq 5$, Wang \cite{Wang_third_line_adams} determined the May spectral sequence representatives of the generators. The result is summarized in the following table.

\begin{longtable}{| c | c | c | c |}
 \hline
 Generator & MSS Representation & Inner Degree & Range of indices\\ 
 \hline
 $h_i h_j h_k$ & $h_{1,i} h_{1,j} h_{1,k}$ & $q(p^i+p^j+p^k)$ &   $k-4 \geq j-2 \geq i \geq 0$ \\
 \hline 
 $a_0 h_i h_j$ & $a_0 h_{1,i} h_{1,j}$ & $q(p^i+p^j) +1$ &  $j-2 \geq i \geq 1$ \\
 \hline
 $a_0^2 h_i$ & $a_0^2 h_{1,i}$ & $qp^i+2$ & $i \geq 1$ \\
 \hline
 $a_0^3$ & $a_0^3$ & $3$ &  \\
 \hline 
 $b_i h_j$ & $b_{1,i} h_{1,j}$ & $q(p^{i+1}+p^j)$ & $i,j \geq 0, j \neq i+2$ \\
 \hline  
 $a_0 b_i$ & $a_0 b_{1,i}$ & $qp^{i+1}+1$ & $i \geq 1$ \\
 \hline   
 $g_i h_j$ & $h_{2,i} h_{1,i} h_{1,j}$ & $q(2p^i+p^{i+1}+p^j)$ & $j \neq i+2,i,i-1,$ \\
  &  &  & and $i,j \geq 0$ \\
 \hline   
 $g_i a_0$ & $h_{2,i} h_{1,i} a_0$ & $q(2p^i+p^{i+1})+1$ & $i \geq 1$ \\
 \hline    
 $k_i h_j$ & $h_{2,i}h_{1,i+1}h_{1,j}$ & $q(p^i+2p^{i+1}+p^j)$ & $j \neq i+2, i\pm 1,i,$ \\
   &  &  & and $i,j \geq 0$ \\
 \hline   
 $k_i a_0$ & $h_{2,i} h_{1,i+1} a_0$ & $q(p^i+2p^{i+1})+1$ & $i \geq 1$ \\
 \hline  
 $a_1 h_0 h_j$ & $a_1 h_{1,0} h_{1,j}$ & $q(2+p^j)+1$ & $j \geq 2$ \\
 \hline 
 $h_{3,i} g_i$ & $h_{3,i} h_{2,i} h_{1,i}$ & $q(3p^{i}+ 2p^{i+1}+p^{i+2})$ & $i \geq 0$ \\
 \hline  
 $a_2 k_0$ & $a_2 h_{2,0} h_{1,1}$ & $q(2+3p)+1$ &  \\
 \hline   
 $h_{2,i} g_{i+1}$ & $h_{2,i} h_{2,i+1} h_{1,i+1}$ & $q(p^{i}+ 3p^{i+1}+p^{i+2})$ & $i \geq 0$ \\
 \hline 
 $a_1 g_0$ & $a_1 h_{2,0} h_{1,0}$ & $q(3+p)+1$ &  \\
 \hline  
 $h_{3,i} h_{i+2} h_i$ & $h_{3,i} h_{1,i+2} h_{1,i}$ & $q(2p^i+p^{i+1}+2p^{i+2})$ &  $i \geq 0$ \\
 \hline  
 $h_{3,i} k_{i+1}$ & $h_{3,i} h_{2,i+1} h_{1,i+2}$ & $q(p^i+2p^{i+1}+3p^{i+2})$ &  $i \geq 0$ \\
 \hline  
 $a_1^2 h_0$ & $a_1^2 h_{1,0}$ & $3q+2$ &  \\
 \hline 
 $b_{2,i} h_{i+1}$ & $b_{2,i} h_{1,i+1}$ & $q(2p^{i+1}+p^{i+2})$ & $i \geq 0$ \\
 \hline  
 $b_{2,i} h_{i+2}$ & $b_{2,i} h_{1,i+2}$ & $q(p^{i+1}+2p^{i+2})$ & $i \geq 0$ \\
 \hline   
\caption{\label{table: third line basis} A $\mathbb{F}_p$-basis of $Ext_{\mathcal{A}_*}^{3,*}$}
\end{longtable}

We can compute $d_2^{Adams}$ for the basis elements in Table \ref{table: first and second line basis} via computing $d_2^{alg}$ of their corresponding elements.  For simplicity, we only list the nontrivial $d_2^{alg}$ differentials here.

\begin{prop}
\label{prop: first and second line alg d2}
Let $p$ be an odd prime.
Amongst the elements in the algebraic Novikov spectral sequence that corresponds to the first and second line basis listed in Table \ref{table: first and second line basis}, all nontrivial $d_2^{alg}$'s are summarized as follows.
Here, the equations hold after modding out lower May filtration terms.

\begin{enumerate}
  \item $d_2^{alg} (t_1^{p^i}) = q_0 b_{1,i-1}$, for $i>0$.
  \item $d_2^{alg} (p t_1^{p^i}) =  q_0^2 b_{1,i-1}$, $i \geq 1$.
  \item $d_2^{alg}(t_2^{p^{i}}\otimes t_1^{p^{i}}) =  q_0 b_{2,i-1}\otimes t_1^{p^{i}}$, $i \geq 1$.
  \item $d_2^{alg} (t_2 \otimes t_1) = - q_1 b_{1,0} \otimes t_1$.
  \item $d_2^{alg}(t_2^{p^{i}}\otimes t_1^{p^{i+1}}) =  q_0 b_{2, i-1} \otimes t_1^{p^{i+1}}$, $i \geq 1$.
  \item $d_2^{alg} (t_1^{p^i} \otimes t_1^{p^j}) = q_0 b_{1,i-1} \otimes t_1^{p^j} - t_1^{p^i} \otimes q_0 b_{1,j-1}$, $j-2 \geq i \geq 1$.
\end{enumerate}

\end{prop}

\begin{proof}
Analogous to Proposition \ref{prop: compute alg d2 first example}, all of the results are computed directly from the construction of the cobar complex. 
\end{proof}

Then, we can recover the $d_2^{Adams}$ results on the first two lines directly from Proposition \ref{prop: first and second line alg d2}.

\begin{thm}[Liulevicius\cite{Liulevicius_factorization}, Shimada-Yamanoshita \cite{Shimada_Yamanoshita}, Miller-Ravenel-Wilson \cite{Miller_Ravenel_Wilson}, Zhao-Wang \cite{Zhao_Wang_adams}]
\label{thm: second line adams d2}
Amongst the first and second line basis in Table \ref{table: first and second line basis}, all nontrivial Adams $d_2$ differentials can be summarized as follows. 
\begin{enumerate}
  \item $d_2^{Adams} (h_i) = a_0 b_{i-1}$, $i \geq 1$. 
  \item $d_2^{Adams} (a_0 h_i) = a_0^2 b_{i-1}$, $i \geq 1$.
  \item $d_2^{Adams} (g_i) = a_0 b_{2, i-1} h_i$, $i \geq 1$.
  \item $d_2^{Adams} (g_0) = - a_1 b_0 h_0$.
  \item $d_2^{Adams} (k_i) = a_0 b_{2, i-1} h_{i+1}$, $i \geq 1$.
  \item $d_2^{Adams} (h_i h_j) = a_0 b_{i-1} h_j - h_i a_0 b_{j-1}$, $j-2 \geq i \geq 1$.  
\end{enumerate}

\end{thm}

Similarly, we can compute $d_2^{Adams}$ for the third line basis via computing $d_2^{alg}$ of their corresponding elements.  For simplicity, we only list the nontrivial differentials here.

\begin{prop}
\label{prop: third line alg d2}
Let $p \geq 5$ be an odd prime.  Amongst the elements in the algebraic Novikov spectral sequence that corresponds to the third line basis listed in Table \ref{table: third line basis}, all nontrivial $d_2^{alg}$'s are summarized as follows.
Here, the equations hold after modding out lower May filtration terms.

\begin{enumerate}
  \item $d_2^{alg} (t_1^{p^i} \otimes t_1^{p^j} \otimes t_1^{p^k}) = q_0 b_{1,i-1} \otimes t_1^{p^j} \otimes t_1^{p^k} - t_1^{p^i} \otimes q_0 b_{1,j-1} \otimes t_1^{p^k} + t_1^{p^i} \otimes t_1^{p^j} \otimes q_0 b_{1,k-1}$, for $k-4 \geq j-2 \geq i \geq 1$.
  \item $d_2^{alg} (q_0 t_1^{p^i} \otimes t_1^{p^j}) =  q_0^2 b_{1,i-1} \otimes t_1^{p^j} - q_0^2 t_1^{p^i}  \otimes b_{1,j-1}$, for $j-2 \geq i \geq 1$.
  \item $d_2^{alg} (q_0^2 t_1^{p^i}) = q_0^3 b_{1,i-1}$, for $i \geq 1$.
  \item $d_2^{alg} (b_{1,i} \otimes t_1^{p^j}) = q_0 b_{1,i} b_{1,j-1}$, for $i \geq 0, j \geq 1, j \neq i+2$. 
  \item $d_2^{alg} (t_2^{p^i} \otimes t_1^{p^i} \otimes t_1^{p^j}) = q_0 b_{2, i-1} \otimes t_1^{p^i} \otimes t_1^{p^j}$, for $i,j \geq 1, j \neq i+2,i,i-1$.
  \item $d_2^{alg} (t_2 \otimes t_1 \otimes t_1^{p^j}) = - q_1 b_{1,0} \otimes t_1 \otimes t_1^{p^j} + t_2 \otimes t_1 \otimes q_0 b_{1, j-1}$, for $j > 0, j \neq 2$.
  \item $d_2^{alg}(q_0 t_2^{p^i} \otimes t_1^{p^i}) = q_0^2 b_{2,i-1} \otimes t_1^{p^i}$, for $i \geq 1$.
  \item $d_2^{alg} (t_2^{p^i} \otimes t_1^{p^{i+1}} \otimes t_1^{p^j}) =  q_0 b_{2, i-1} \otimes t_1^{p^{i+1}} \otimes t_1^{p^j}$, for $i,j \geq 1, j \neq i+2, i\pm 1,i$.  
  \item $d_2^{alg} (q_0 t_2^{p^{i}} \otimes t_1^{p^{i+1}}) = q_0^2 b_{2,i-1} \otimes t_1^{p^{i+1}}$, for $i \geq 1$.
  \item $d_2^{alg}(t_3^{p^{i}} \otimes t_2^{p^{i}} \otimes t_1^{p^{i}}) = q_0 b_{3, i-1} \otimes t_2^{p^{i}} \otimes t_1^{p^{i}}$, for $i \geq 1$.
  \item $d_2^{alg}(t_3 \otimes t_2 \otimes t_1) = - q_1 b_{2,0} \otimes t_2 \otimes t_1$.  
  \item $d_2^{alg} (t_2^{p^i} \otimes t_2^{p^{i+1}} \otimes t_1^{p^{i+1}}) = q_0 b_{2, i-1} \otimes t_2^{p^{i+1}} \otimes t_1^{p^{i+1}} - t_2^{p^i} \otimes q_0 b_{2, i} \otimes t_1^{p^{i+1}}$, for $i \geq 1$.
  \item $d_2^{alg} (q_1 t_2 \otimes t_1) = - q_1^2 b_{1,0} \otimes t_1$.
  \item $d_2^{alg}(t_3^{p^{i}} \otimes t_1^{p^{i+2}} \otimes t_1^{p^{i}}) = q_0 b_{3, i-1} \otimes t_1^{p^{i+2}} \otimes t_1^{p^{i}}$, for $i \geq 1$.
  \item $d_2^{alg}(t_3 \otimes t_1^{p^2} \otimes t_1) = - q_1 b_{2,0} \otimes t_1^{p^2} \otimes t_1$.
  \item $d_2^{alg}(t_3^{p^{i}} \otimes t_2^{p^{i+1}} \otimes t_1^{p^{i+2}}) = q_0 b_{3, i-1} \otimes t_2^{p^{i+1}} \otimes t_1^{p^{i+2}}$, for $i \geq 1$.
\end{enumerate}

\end{prop}

Then, we can recover the following result directly from Proposition \ref{prop: third line alg d2}.

\begin{thm}[Wang \cite{Wang_third_line_adams}]
\label{thm: third line adams d2}
Let $p \geq 5$ be an odd prime.  Amongst the third line basis in Table \ref{table: third line basis}, all nontrivial Adams $d_2$ differentials can be summarized as follows. 
\begin{enumerate}
  \item $d_2^{Adams} (h_i h_j h_k) = a_0 b_{i-1} h_{j} h_{k} - a_0 h_{i} b_{j-1} h_{k} + a_0 h_{i} h_{j} b_{k-1}$, $k-4 \geq j-2 \geq i \geq 1$.
  \item $d_2^{Adams} (a_0 h_i h_j) = a_0^2 b_{i-1} h_j - a_0^2 h_i b_{j-1}$, $j-2 \geq i \geq 1$.
  \item $d_2^{Adams} (a_0^2 h_i) = a_0^3 b_{i-1}$, $i \geq 1$.
  \item $d_2^{Adams} (b_i h_j) = a_0 b_i b_{j-1}$, $i \geq 0, j \geq 1, j \neq i+2$.
  \item $d_2^{Adams} (g_i h_j) = a_0 b_{2, i-1} h_i h_j$, $i,j \geq 1, j \neq i+2,i,i-1$.
  \item $d_2^{Adams} (g_0 h_j) = - a_1 b_0 h_0 h_j + a_0 g_0 b_{j-1}$, $j > 0, j \neq 2$.
  \item $d_2^{Adams} (g_i a_0) = a_0^2 b_{2,i-1} h_i$, $i \geq 1$.
  \item $d_2^{Adams} (k_i h_j) = a_0 b_{2, i-1} h_{i+1} h_j$, $i,j \geq 1, j \neq i+2, i\pm 1,i$.
  \item $d_2^{Adams} (k_i a_0) = a_0^2 b_{2,i-1} h_{i+1}$, $i \geq 1$.
  \item $d_2^{Adams} (h_{3,i} g_i) = a_0 b_{3, i-1} g_i$, $i \geq 1$.
  \item $d_2^{Adams} (h_{3,0} g_0) = - a_1 b_{2,0} g_0$.
  \item $d_2^{Adams} (h_{2,i} g_{i+1}) = a_0 b_{2, i-1} g_{i+1} - a_0 h_{2,i} k_i$, $i \geq 1$. 
  \item $d_2^{Adams} (a_1 g_0) = - a_1^2 b_0 h_0$.
  \item $d_2^{Adams} (h_{3,i} h_{i+2} h_i) = a_0 b_{3, i-1} h_{i+2} h_i$, $i \geq 1$.
  \item $d_2^{Adams} (h_{3,0} h_{2} h_0) = - a_1 b_{2,0} h_{2} h_0$.
  \item $d_2^{Adams} (h_{3,i} k_{i+1}) = a_0 b_{3, i-1} k_{i+1}$, $i \geq 1$.
\end{enumerate}

\end{thm}

\end{document}